\documentclass{aptpub}

\authornames{Carl Graham} 
\shorttitle{Convergence of multi-class systems of fixed sizes}

\usepackage{amsfonts,amsmath} 

\newcommand{\EE}{\kern1pt\mathbf{E}\kern1pt}
\newcommand{\PP}{\kern1pt\mathbf{P}\kern1pt}
\newcommand\One{{\mathrm{1} \kern -0.27em \mathrm{I}}}

\renewcommand{\today}{3 February 2009} 

\begin{document}

\title{Convergence of multi-class systems of fixed possibly infinite sizes}

\author[{\'E}cole Polytechnique, CNRS]{Carl Graham}
\address{CMAP, {\'E}cole Polytechnique, CNRS, 91128 Palaiseau France.}
\email{carl@cmapx.polytechnique.fr}

\begin{abstract}
Multi-class systems having possibly both finite and infinite classes are investigated under a natural partial exchangeability assumption. It is proved that the conditional law of such a system, given the vector of the empirical measures of its finite classes and directing measures of its infinite ones (given by the de Finetti Theorem), corresponds to sampling \emph{independently} from each class, without replacement from the finite classes and i.i.d.\ from the directing measure for the infinite ones. The equivalence between the convergence of multi-exchangeable systems with fixed class sizes and the convergence of the corresponding vectors of measures is then established.
\end{abstract}



\section{Introduction}

Kallenberg~\cite{Kallenberg}, Kingman~\cite{Kingman}, Diaconis and Freedman~\cite{Diaconis},
and Aldous~\cite{Aldous} are among many studies of exchangeable random variables (r.v.)\ with Polish state spaces, a fundamental topic in many fields of probability and statistics. Nevertheless, many models in stratified sampling, statistical mechanics, chemistry,
communication networks, biology, etc., actually involve \emph{varied} classes of similar objects, which we call ``particles''. Some examples can be found in Graham~\cite{Graham:92, Graham:00} and Graham and Robert~\cite{GrahamRobert}. 

Several classical results for exchangeable systems are extended to such multi-class particle systems in Graham~\cite{Graham:08}, under a natural partial exchangeability assumption called \emph{multi-exchangeability}. It is notably proved that the convergence in law of a family of \emph{finite} systems, with limit a system with \emph{infinite} class sizes, is equivalent to the convergence in law of the corresponding family of vectors of the empirical measures within each class, with limit the vector of the directing measures of each limit class (which is exchangeable, and the directing measure is given by the de Finetti theorem). This result allows use of compactness-uniqueness techniques on the measure vectors for convergence proofs, extending those for propagation of chaos proofs for exchangeable systems introduced by Sznitman~\cite{Sznitman} and developed among others by M{\'e}l{\'e}ard~\cite{Meleard:96} and Graham~\cite{Graham:92,Graham:00}.
We refer to \cite{Graham:08} for further motivation.

In the present paper, we extend this convergence result to multi-exchangeable systems with \emph{fixed} class sizes, which may be both finite and infinite, by showing that their convergence in law is equivalent to the convergence in law for the vectors of measures with components the empirical measures for finite classes and the directing measure for infinite classes (given by the de Finetti Theorem). 
We first extend another result of \cite{Graham:08}, and prove that the conditional law of a multi-echangeable system, given the vector of measures defined above, corresponds to sampling \emph{independently} without replacement from \emph{each} finite class and i.i.d.\ from the directing measure for \emph{each} infinite class, the remarkable fact being independence between \emph{different} classes.

These new results are related to, and combine well with, those of \cite{Graham:08}. For instance, results in \cite{Graham:08} can yield a tractable large size limit for a finite model of interest, one hopes that the long-time behavior of the finite model is well approximated by that of the limit, and the new results can help us study all the long-time behaviors and their relationship. They can be likewise used for \emph{e.g.} fluid limits or diffusion approximations. 

All state spaces $\mathcal{S}$ are Polish, and the weak topology is used for the space of probability measures $\mathcal{P}(\mathcal{S})$ which is then also Polish, as are products of Polish spaces. For $k \ge1$ we denote by $\Sigma(k)$ the set of permutations of $\{1,\ldots,k\}$, and by $\Sigma(\infty)$ the set of permutations of $\{1,2, \dots\}$ with \emph{finite} support. 

\section{Some combinatorial facts}

\begin{theorem}
\label{combfact}
Let $1 \le k \le N$ and $(N)_k = N(N-1)\cdots(N-k+1)$. Sampling $k$ times without replacement among $N$ possibly not distinct points $x_1, \dots, x_N \in \mathcal{S}$ corresponds to sampling from the law
\[
\lambda^{N,k}(x_1, \ldots, x_N)
=
{1 \over (N)_k} 
\sum_{ \substack{ 1 \le n_1,\ldots, n_k \le N \\ \mathrm{distinct}}}
\delta_{ (x_{n_1}, \ldots, x_{n_k}) }
\in \mathcal{P}(\mathcal{S}^k)
\]
which is a continuous function (for the weak and the total variation topologies) of \[
\lambda^{N,1}(x_1, \ldots, x_N) = {1 \over N} \sum_{n=1}^N \delta_{x_n} \in \mathcal{P}(\mathcal{S})\,.
\]
More precisely, $\lambda^{N,k}(x_1, \ldots, x_N)$ can be written as a sum of continuous linear functions of
$\lambda^{N,1}(x_1, \ldots, x_N)^{\otimes j} = \left({1 \over N} \sum_{n=1}^N \delta_{x_n}\right)^{\otimes j}$ for $1\le j \le k$ by an exclusion-inclusion formula.
Sampling $k$ times with replacement corresponds to using the law
$
\lambda^{N,1}(x_1, \ldots, x_N)^{\otimes k}
=
\left({1 \over N} \sum_{n=1}^N \delta_{x_n}\right)^{\otimes k}$.
\end{theorem}

\begin{proof}
The statements about the laws used for sampling without and with replacement are obvious.
Sampling from $\lambda^{N,k}(x_1, \ldots, x_N)$ corresponds to sampling $k$ times without replacement from the atoms of ${1 \over N} \sum_{n=1}^N \delta_{x_n}$ counted with their multiplicities, and this fact clearly implies the continuity statement. More precisely,
\begin{align}
\label{deckj}
\left({1 \over N} \sum_{n=1} \delta_{x_n}\right)^{\otimes k}
&=
{1 \over N^k}\sum_{1 \le n_1,\ldots, n_k\le N} \delta_{ (x_{n_1}, \ldots, x_{n_k}) }
\nonumber \\
&=
{(N)_k \over N^k}\, \lambda^{N,k}(x_1, \ldots, x_N)
+ {1 \over N^k}\sum_{j=1}^{k-1}
\sum_{ \substack{ 1 \le n_1,\ldots, n_k\le N \\ \mathrm{Card}\{n_1, \ldots, n_k\} = j }}
\delta_{ (x_{n_1}, \ldots, x_{n_k}) }
\end{align}
where the term of index $j$ in the sum
is a continuous linear function of $\lambda^{N,j}(x_1, \ldots, x_N)$, and since $j\le k-1$ we
conclude by recurrence over $k\ge1$.
\end{proof}

The following result quantifies the difference between sampling with and without replacement
in variation norm $\Vert \mu \Vert = \sup\{\,\langle \phi, \mu \rangle : \Vert \phi \Vert_\infty \le 1\,\}$.
See also \cite[Prop.~5.6 p.~39]{Aldous} and \cite[Theorem~13 p.~749]{Diaconis}.

\begin{theorem}
\label{combfactvar}
With the notations of Theorem~\ref{combfact}, we have
\[
\left\Vert 
\lambda^{N,k}(x_1, \ldots, x_N)
-
\lambda^{N,1}(x_1, \ldots, x_N)^{\otimes k}
\right\Vert
\le 2 {N^k - (N)_k  \over N^k}  \le {k(k-1) \over N}
\]

and the first inequality is an equality if and only if the $x_1$, \dots\,, $x_N$ are distinct.
\end{theorem}

\begin{proof}
Equation \eqref{deckj} yields the first inequality for the variation norm, and the condition for it to be an equality. The second inequality follows by bounding $N^k - (N)_k$ by counting $k(k-1)/2$ possible positions for two identical indices with $N$ choices and $N^{k-2}$ choices for the other $k-2$ positions. 
\end{proof}

\section{Multi-exchangeable systems}

\subsection{Reminders on exchangeable systems}
\label{ss:rem}

Let $N \in \mathbb{N} = \{0,1,\dots\}$ be fixed. A \emph{finite} system $(X_n)_{1\le n \le N}$ of random variables (r.v.) with state space $\mathcal{S}$ is \emph{exchangeable} if 
\[
\mathcal{L}(X_{\sigma(1)}, \ldots, X_{\sigma(N)}) 
= \mathcal{L}(X_1, \ldots, X_N)\,,
\qquad
\forall \sigma \in \Sigma(N)\,.
\]
The \emph{empirical measure} of the system is the random probability measure 
\begin{equation}
\label{emplaw}
\Lambda^N 
= \lambda^{N,1}(X_1, \ldots, X_N)
= {1\over N} \sum_{n=1}^{N} \delta_{X_n} 
\end{equation}
with samples in $\mathcal{P}(\mathcal{S})$. 
The conditional law of such an exchangeable system given its empirical meausure $\Lambda^N$ is $\lambda^{N,N}(X_1, \ldots, X_N)$, see Theorem~\ref{combfact} for the definition of $\lambda^{N,N}$ and \emph{e.g.}\ Aldous~\cite[Lemma 5.4 p.~38]{Aldous} for the result.

An \emph{infinite} system $(X_n)_{n \ge 1}$ is \emph{exchangeable} if every 
finite subsystem $(X_n)_{1 \le n \le N}$ is exchangeable. The de Finetti Theorem, 
see \emph{e.g.} \cite{Kallenberg, Kingman,Diaconis,Aldous}, states that such a system is a mixture of i.i.d.\ sequences, and precisely that its law is of the form 
\[
\int P^{\otimes \infty} \mathcal{L}_{\Lambda^\infty}(dP)
\]
where $\mathcal{L}_{\Lambda^\infty}$ is the law of the random probability measure
\begin{equation}
\label{dirmeas}
\Lambda^\infty 
= \lim_{N \to \infty}\lambda^{N,1}(X_1, \ldots, X_N)
= \lim_{N \to \infty} {1\over N} \sum_{n=1}^{N} \delta_{X_n}
\;\;
\textrm{a.s.}
\end{equation}
called the \emph{directing measure} of the system. 
The conditional law of such a system given $\Lambda^\infty$ 
corresponds to i.i.d.\ draws from $\Lambda^\infty$.

\subsection{Multi-class systems}

In order to consider finite and infinite systems simultaneously, we will take class sizes in $\mathbb{N} \cup \{\infty\}$, and be redundant for clarity.
Let $C \ge 1$ and $N_i \in \mathbb{N} \cup \{\infty\}$ and state spaces $\mathcal{S}_i$ be fixed for $1 \le i \le C$, and consider a \emph{multi-class} system
\begin{equation}
\label{eq:mcys}
(X_{n,i})_{1\le n \le N_i,\, 1\le i \le C}\,,
\qquad
X_{n,i}
\textrm{ with state space $\mathcal{S}_i$},
\end{equation}
where the r.v.\ $X_{n,i}$ is the $n$-th particle, or object, of class $i$, and $1\le n \le \infty$ is interpreted as $n \ge1$.
We say that \eqref{eq:mcys} is a \emph{multi-exchangeable} system if its law is invariant under finite permutations of the particles \emph{within} classes, and precisely if
\[
\mathcal{L}\bigl( (X_{\sigma_i(n),i})_{1\le n \le N_i,\, 1\le i \le C} \bigr)
=
\mathcal{L}\bigl( (X_{n,i})_{1\le n \le N_i,\, 1\le i \le C} \bigr)\,,
\qquad
\forall \sigma_i \in \Sigma(N_i)\,.
\]
This natural assumption means that particles of a class are
statistically indistinguishable, and obviously
implies that $(X_{n,i})_{1\le n \le N_i}$ is exchangeable for all $i$.
It is sufficient to check that it is true 
when all $\sigma_i$ but one are the identity.

For a multi-exchangeable system, a fundamental quantity is the random vector of probability measures, with
samples in $\mathcal{P}(\mathcal{S}_1) \times \cdots \times \mathcal{P}(\mathcal{S}_C)$,
given by
\begin{equation}
  \label{eq:emmeve}
(\Lambda^{N_i}_i)_{1\le i \le C}\,,
\qquad
\Lambda^{N_i}_i 
=
\left\{
\begin{array}{l}
\textrm{the empirical measure given by \eqref{emplaw} if } N_i<\infty
\\
\textrm{the directing measure given by \eqref{dirmeas} if } N_i=\infty\,.
\end{array}
\right.
\end{equation}

The following extends the results in Section~\ref{ss:rem}, as well as 
Graham~\cite[Theorem~1]{Graham:08} and Aldous~\cite[Cor.~3.9 p.~25]{Aldous} in which respectively $N_i<\infty$ and $N_i=\infty$ for all $i$. The remarkable fact is conditional independence between \emph{different} classes.

\begin{theorem}
\label{suffstat}
The conditional law of a multi-exchangeable system \eqref{eq:mcys} given the random measure vector $(\Lambda^{N_i}_i)_{1\le i \le C}$ in \eqref{eq:emmeve} corresponds to drawing \emph{independently} for each class~$i$, if $N_i<\infty$ from $\lambda^{N_i,N_i}(X_{n,1},\dots,X_{n,N_i})$ given in Theorem~\ref{combfact}, \emph{i.e}, without replacement from the atoms of the empirical measure $\Lambda^{N_i}_i = {1\over N_i} \sum_{n=1}^{N_i} \delta_{X_{n,i}}$ counted with their multiplicities, if $N_i=\infty$ in i.i.d.\ fashion from the directing measure $\Lambda^\infty_i$.
\end{theorem}

\begin{proof}
Let $1 \le k \le M<\infty$ be arbitrary and, for $1\le i \le C$,
\begin{equation}
\label{kimi}
k_i = M_i = N_i \text{ if } N_i<\infty\,,
\qquad
k_i = k \text{ and } M_i = M \text{ if } N_i=\infty\,.
\end{equation}
Since $\Lambda^{N_i}_i$ does not change if one applies a permutation to $(X_{n,i})_{1 \le n \le M_i}$, see \eqref{eq:emmeve}, multi-exchangeability implies that for all
$f_i \in C_b(\mathcal{S}_i^{k_i},\mathbb{R})$ 
and 
$g \in C_b(\mathcal{P}(\mathcal{S}_1) \times \cdots \times \mathcal{P}(\mathcal{S}_C),\mathbb{R})$, with the notation in Theorem~\ref{combfact},
\begin{eqnarray}
\label{cenfor}
&&\kern-6.5mm
\EE\left[
g\bigl((\Lambda^{N_i}_i)_{1 \le i \le C}\bigr)
\prod_{i=1}^C 
f_i(X_{1,i}, \ldots X_{k_i,i})
\right]
\nonumber\\
&&\kern-5mm{}
= {1 \over M_1 !}\sum_{\sigma_1 \in \Sigma(M_1)} 
\cdots {1 \over M_C !}\sum_{\sigma_C \in \Sigma(M_C)}
\EE\left[
g\bigl((\Lambda^{N_i}_i)_{1 \le i \le C}\bigr)
\prod_{i=1}^C 
f_i(X_{\sigma_i(1),i}, \ldots X_{\sigma_i(k_i),i})
\right]
\nonumber\\
&&\kern-5mm{}
=
\EE\!\left[
g\bigl((\Lambda^{N_i}_i)_{1 \le i \le C}\bigr)
\prod_{i=1}^C  
{1 \over M_i !} \sum_{\sigma \in \Sigma(M_i) }
f_i(X_{\sigma(1),i}, \ldots, X_{\sigma(k_i),i})
\right]
\nonumber\\
&&\kern-5mm{}
=
\EE\!\left[
g\bigl((\Lambda^{N_i}_i)_{1 \le i \le C}\bigr)
\prod_{i=1}^C  
\left\langle f_i,
\lambda^{M_i,k_i}(X_{1,i},\dots,X_{M_i,i})
\right\rangle
\right].
\end{eqnarray}
Considering \eqref{kimi}, if $N_i < \infty$ then 
$\lambda^{M_i,k_i}(X_{1,i},\dots,X_{N_i,i}) = \lambda^{N_i,N_i}(X_{1,i},\dots,X_{N_i,i})$, and if $N_i = \infty$ then we let $M$ go to infinity and use Theorem~\ref{combfactvar}, \eqref{dirmeas}, continuity, boundedness, and dominated convergence, and we obtain
\begin{eqnarray}
\label{intfor}
&&\kern-11mm
\EE\left[
g\bigl((\Lambda^{N_i}_i)_{1 \le i \le C}\bigr)
\prod_{i=1}^C 
f_i(X_{1,i}, \ldots X_{k_i,i})
\right]
\nonumber\\
&&\kern-9mm{}
=
\EE\!\left[
g\bigl((\Lambda^{N_i}_i)_{1 \le i \le C}\bigr)
\prod_{i=1}^C
\left\langle f_i,
\One_{N_i < \infty }\lambda^{N_i,N_i}(X_{1,i},\dots,X_{N_i,i})
+ \One_{N_i = \infty } (\Lambda^{N_i}_i)^{\otimes k}
\right\rangle
\right]
\end{eqnarray}
and we conclude using that $k \ge1$ and $g,f_i\in C_b$ are arbitrary, the spaces are Polish, $\lambda^{N_i,N_i}(X_{1,i},\dots,X_{N_i,i})$ depends measurably on $\Lambda^{N_i}_i = \lambda^{N_i,1}(X_{1,i},\dots,X_{N_i,i})$ (see Theorem~\ref{combfact}), and the characteristic property of conditional expectation.
\end{proof}

Thus, the classes of a multi-exchangeable system are \emph{conditionally independent} given $(\Lambda^{N_i}_i)_{1\le i \le C}$, with conditional laws as for an exchangeable system and depending on whether the class is finite or infinite. 
Given $(\Lambda^{N_i}_i)_{1\le i \le C}$ \emph{no further information} can be attained on the law of the multi-exchangeable system by further observation, in particular involving r.v.\ in different classes.
A statistical interpretation is that the empirical measure vector is a 
\emph{sufficient statistic} for the law of the system,
the family of all such laws being trivially parametrized by the laws themselves.

We now extend the convergence results for exchangeable systems in Kallenberg~\cite[Theorem~1.3 p.~25]{Kallenberg}, see also Aldous~\cite[Prop.~7.20 p.~55]{Aldous} when $N_1<\infty$, to a family of multi-exchangeable systems, of fixed possibly both finite and infinite class sizes. 

\begin{theorem}
Let $C \ge 1$ and $N_i \in \mathbb{N} \cup \{\infty\}$ and state spaces $\mathcal{S}_i$ be fixed for $1 \le i \le C$.
For $r \in \mathbb{R}_+ \cup \{\infty\}$ let 
\[
(X^{r}_{n,i})_{1\le n \le N_i,\, 1\le i \le C}\,,
\qquad
X^{r}_{n,i}
\textrm{ with state space $\mathcal{S}_i$}\,,
\]
be \emph{multi-exchangeable} systems as in \eqref{eq:mcys}, and $(\Lambda^{N_i,r}_i)_{1 \le i \le C}$ be as in \eqref{eq:emmeve}. Then
\[
\lim_{r \to \infty} (X^{r}_{n,i})_{1\le n \le N_i,\, 1\le i \le C}
= (X^{\infty}_{n,i})_{n \ge 1\,, 1\le i \le C}
\;\; \textrm{in law} 
\]
if and only if
\[
\lim_{r \to \infty} (\Lambda^{N_i,r}_i)_{1 \le i \le C} = (\Lambda^{N_i,\infty}_i)_{1 \le i \le C}
\;\; \textrm{in law.}
\]
\end{theorem}

\begin{proof}
With the notation \eqref{kimi}, let $k \ge1$ and  
$f_i \in C_b(\mathcal{S}_i^{k_i},\mathbb{R})$ for $1\le i \le C$ be arbitrary. Taking $g=1$ in \eqref{intfor}, or conditioning first on $(\Lambda^{N_i,r}_i)_{1 \le i \le C}$ and using Theorem~\ref{suffstat}, yields 
\begin{eqnarray}
\label{condlam}
&&\kern-6mm
\EE\left[
\prod_{i=1}^C 
f_i(X^{r}_{1,i}, \ldots X^{r}_{k_i,i})
\right]
\nonumber\\
&&{}
=
\EE\!\left[
\prod_{i=1}^C
\left\langle f_i,
\One_{N_i < \infty }\lambda^{N_i,N_i}(X^{r}_{1,i},\dots,X^{r}_{N_i,i})
+ \One_{N_i = \infty } (\Lambda^{N_i,r}_i)^{\otimes k}
\right\rangle
\right]
\end{eqnarray}
and hence if
$
\lim_{r \to \infty} (\Lambda^{N_i,r}_i)_{1 \le i \le C} = (\Lambda^{N_i,\infty}_i)_{1 \le i \le C}
$
in law then
\[
\lim_{r \to \infty} 
\EE\left[
\prod_{i=1}^C 
f_i(X^{r}_{1,i}, \ldots X^{r}_{k_i,i})
\right]
= 
\EE\left[
\prod_{i=1}^C 
f_i(X^{\infty}_{1,i}, \ldots X^{\infty}_{k_i,i})
\right]
\]
using the continuity result in Theorem~\ref{combfact}, which proves the ``if'' part of the theorem since $k \ge1$ and $f_i$ are arbitrary and the state spaces are Polish. The converse follows similarly from \eqref{condlam} with functions $f_i$ depending only on the first variable.
\end{proof}


\end{document}